\documentclass[12pt,a4paper,reqno]{amsart}

\usepackage{amssymb}
\usepackage{amscd}
\usepackage[pdftex,pdfpagelabels]{hyperref}
\usepackage{enumerate}
\usepackage{graphicx}
\usepackage{siunitx}
\usepackage{tikz-cd}
\usepackage{stix}
\usepackage{bm}
\DeclareMathAlphabet\mathbfcal{LS2}{stixcal}{b}{n}
\numberwithin{equation}{section}

\usepackage{mathtools}%                  http://www.ctan.org/pkg/mathtools
\usepackage[tableposition=top]{caption}% http://www.ctan.org/pkg/caption
\usepackage{booktabs,dcolumn}%           http://www.ctan.org/pkg/dcolumn + http://www.ctan.org/pkg/booktabs

\DeclareFontFamily{OT1}{rsfs}{}
\DeclareFontShape{OT1}{rsfs}{n}{it}{<-> rsfs10}{}
\DeclareMathAlphabet{\mathscr}{OT1}{rsfs}{n}{it}

\theoremstyle{plain}

\newtheorem{theorem}{Theorem}[section]

\newtheorem{lemma}[theorem]{Lemma}

\newtheorem{conjecture}[theorem]{Conjecture}

\newtheorem{question}[theorem]{Question}

\theoremstyle{definition}

\newtheorem{remark}[theorem]{Remark}

\renewcommand\P{\mathbf{P}}
\newcommand\E{\mathbf{E}}
\newcommand\Var{\mathbf{Var}}

\newcommand\n{\mathbf{n}}

\newcommand\C{\mathbb{C}}

\newcommand\eps{\varepsilon}

\begin{document}

\title[Convergence of a series of Erd\H{o}s]{The convergence of an alternating series of Erd\H{o}s, assuming the Hardy--Littlewood prime tuples conjecture}

\author{Terence Tao}
\address{UCLA Department of Mathematics, Los Angeles, CA 90095-1555.}
\email{tao@math.ucla.edu}

%\email{}

\subjclass[2020]{11N05, 40A05}

\begin{abstract}  It is an open question of Erd\H{o}s as to whether the alternating series $\sum_{n=1}^\infty \frac{(-1)^n n}{p_n}$ is (conditionally) convergent, where $p_n$ denotes the $n^{\mathrm{th}}$ prime.  By using a random sifted model of the primes recently introduced by Banks, Ford, and the author, as well as variants of a well known calculation of Gallagher, we show that the answer to this question is affirmative assuming a suitably strong version of the Hardy--Littlewood prime tuples conjecture.
\end{abstract}

\maketitle

%%%%%%%%%%%%%%%%%%%%%%%%%%%%%%%%%%%%%%%%%%%%%%%%%

\section{Introduction}

Let $p_n$ denote the $n^{\mathrm{th}}$ prime.  This paper concerns the following question of Erd\H{o}s \cite[p. 203, E7]{guy} (see also \cite[p. 96]{finch}, \cite{erdos}, \cite{en}):

\begin{question}\label{main}  Does the series $\sum_{n=1}^\infty \frac{(-1)^n n}{p_n}$  converge?
\end{question}

Numerically\footnote{Tom\'as Oliveira da Silva, private communication.}, the series appears to converge (rather slowly) to about $-0.052161$; we do not expect the exact value of this sum to have any additional notable properties.  As recorded in \cite[p. 203, E7]{guy}, Erd\H{o}s noted that the variant series $\sum_{n=1}^\infty \frac{(-1)^n n \log n}{p_n}$ is divergent (the summands fail to converge to zero). 
  
The difference between consecutive magnitudes $\frac{n}{p_n}, \frac{n+1}{p_{n+1}}$ of this alternating series is largely controlled by the prime gap $p_{n+1}-p_n$, so the answer to Question \ref{main} revolves around how this prime gap correlates with the parity of $n$.  If such a correlation existed, it would create a bias in the parity of the prime counting function\footnote{The variable $p$ is always understood to range over primes in this paper.} $\pi(x) \coloneqq \sum_{p \leq x} 1$.  Indeed, it is an unpublished observation of Said\footnote{\url{https://mathoverflow.net/questions/313999}} that Question \ref{main} is equivalent to the following.

\begin{question}\label{main-2}  Does the series $\sum_{n=2}^\infty \frac{(-1)^{\pi(n)}}{n \log n}$ converge?
\end{question}

For the convenience of the reader, we establish the equivalence of Question \ref{main} and Question \ref{main-2} in Section \ref{impl}.

In this paper we show that the answer to both of these questions is affirmative assuming a suitably strong version of the Hardy--Littlewood prime tuples conjecture \cite{tuples}.  We use a (very slightly strengthened) formulation from the recent paper \cite{kuperberg}:

\begin{conjecture}[Quantitative Hardy--Littlewood prime tuples conjecture]\label{hlc}  \cite[Conjecture 1.3]{kuperberg}
There exist two absolute constants $\eps>0$ and $C>0$ such that for all $x \geq 10$, all $k \leq (\log \log x)^5$, and all tuples ${\mathcal H} = \{h_1,\dots,h_k\} \subset [0, \log^2 x]$ of distinct integers $h_1,\dots,h_k$, one has
\begin{equation}\label{hl}
\left|\sum_{n \leq x} 1_{\mathcal{P}}(n+h_1) \dots 1_{\mathcal{P}}(n+h_k) - {\mathfrak S}({\mathcal H}) \int_2^x \frac{dy}{\log^k y}\right| \leq C x^{1-\eps}.
\end{equation}
\end{conjecture}

Here $1_{\mathcal{P}}$ is the indicator function of the primes ${\mathcal P}$, ${\mathfrak S}({\mathcal H})$ is the singular series
$$ {\mathfrak S}({\mathcal H})  \coloneqq \prod_p \frac{1-\nu_{\mathcal H}(p)}{(1-1/p)^k},$$
and $\nu_{\mathcal H}(p)$ denotes the number of distinct residue classes modulo $p$ occupied by ${\mathcal H}$.
It is common to restrict \eqref{hl} to the admissible case ${\mathfrak S}({\mathcal H}) > 0$, but this is not necessary since \eqref{hl} is very easy to establish in the non-admissible case ${\mathfrak S}({\mathcal H}) = 0$.
In \cite{kuperberg}, $k$ was restricted to the slightly smaller range $k \leq (\log\log x)^3$, but we will need to increase the exponent $3$ to $5$ for technical reasons. We remark that Conjecture \ref{hlc} has been verified almost surely (even with this enlarged range of parameters) for a certain random model of the primes; see \cite[Theorem 1.3]{bft}.  A version of this random model will also play a key role in the arguments of this paper.

Our main result is then 

\begin{theorem}\label{main-impl}  Assume Conjecture \ref{hlc}.  Then the answer to Question \ref{main-2} (and hence Question \ref{main}) is affirmative.
\end{theorem}

We prove this theorem in Section \ref{impl-2} below.  After an application of the van der Corput $A$-process, the problem reduces to showing that the parity of the number of primes in intervals of the form $(x, x+\lambda \log x]$ are fairly well equidistributed, for moderately large $\lambda$; a typical choice of parameters is $\lambda \asymp (\log\log x)^{4.4}$.  Using Conjecture \ref{hlc}, this then reduces to controlling a certain alternating sum involving the singular series ${\mathfrak S}({\mathcal H})$ over various $k$-tuples ${\mathcal H} = \{h_1,\dots,h_k\}$.  It is tempting to then control the contribution of each $k$ separately and treat the errors by the triangle inequality, in the spirit of a well known calculation of Gallagher \cite{gallagher} which predicts Poissonian statistics for the number of primes in $(x, x+\lambda \log x]$ (which implies in particular that the mean value of $(-1)^{\pi(x+\lambda \log x)-\pi(x)}$ converges to $e^{-2\lambda}$ in the limit $x \to \infty$ for $\lambda$ fixed).   Unfortunately, the range of $k$ needed is a little bit too large to use the strongest available estimates (due to V. Kuperberg \cite{kuperberg}) for mean values of the singular series in this fashion; we need $k$ to be as large as $(\log\log x)^{4.5}$ or so, but quantitative estimates only currently have good error terms in the regime $k \leq (\log\log x)^{0.5}$. However, we can proceed instead by expressing these sums (up to acceptable errors) using the random sieve model $\mathbfcal{S}_z$ introduced in \cite{bft}, effectively replacing the primes by this model.  One then performs some probabilistic calculations on that model to ensure the required equidistribution, taking advantage of cancellations between the contributions of different values of $k$.  The key point is that many of the sifting steps used to build the random sieve model act to reduce any irregularities of distribution in the parity of the number of sifted elements, rather than increase it.

We emphasize that the random sieve model is not used merely to give heuristic support to various statements about the primes, but to actually prove rigorous (albeit conditional) results about the primes themselves.  With this perspective, one can interpret the prime tuples conjecture as an assertion that various statistics concerning the primes are very well approximated by the corresponding statistics of the random sieve model, so that calculations involving the latter can be used to also control the former.

We use the usual asymptotic notation $X = O(Y)$, $X \ll Y$, or $Y \gg X$ to denote the estimate $|X| \leq CY$ for an absolute constant $C$, and $X \asymp Y$ to denote the estimate $X \ll Y \ll X$.

The author is supported by NSF grant DMS-1764034 and by a Simons Investigator Award.  We thank William Banks, Kevin Ford, Dan Goldston, Ofir Gorodetsky, and Kannan Soundararajan for suppying references, and to William Banks for suggesting some additional questions and generalizations to the argument.  We also thank Binbin Hu for a correction.

\section{Equivalence of Question \ref{main} and Question \ref{main-2}}\label{impl}

We now show the equivalence of Question \ref{main} and Question \ref{main-2}.  We will in fact show the more precise relation
\begin{equation}\label{nx}
 \sum_{n \leq x} \frac{(-1)^n n}{p_n} = \frac{1}{2} \sum_{2 \leq m \leq x \log x} \frac{(-1)^{\pi(m)}}{m \log m} + C + o(1)
\end{equation}
for some absolute constant $C$, where $o(1)$ denotes a quantity that goes to zero as $x \to \infty$, from which the equivalence of the two questions clearly follows.

Firstly, by shifting the summation variable $n$ by $1$, we have
$$ \sum_{n \leq x} \frac{(-1)^n n}{p_n} = -\frac{1}{2} + \sum_{n \leq x} \frac{(-1)^{n+1} (n+1)}{p_{n+1}} + o(1).$$
Averaging the left and right-hand sides, we conclude that
$$ \sum_{n \leq x} \frac{(-1)^n n}{p_n} = -\frac{1}{4} + \frac{1}{2} \sum_{n \leq x} \left( \frac{(-1)^n n}{p_n} + \frac{(-1)^{n+1} (n+1)}{p_{n+1}}\right)  + o(1).$$
One could also obtain this identity by writing $(-1)^n = \frac{1}{2} ((-1)^n - (-1)^{n-1})$ and then summing by parts.  We may rewrite the summand on the right-hand side using the identity
\begin{equation}\label{ident}
 \frac{(-1)^n n}{p_n} + \frac{(-1)^{n+1} (n+1)}{p_{n+1}} = \frac{(-1)^n n (p_{n+1}-p_n)}{p_n p_{n+1}} - \frac{(-1)^n}{p_{n+1}}.
\end{equation}
By the alternating series test we have
$$ \sum_{n \leq x} \frac{(-1)^n}{p_{n+1}} = C_0 + o(1)$$
for some absolute constant $C_0$.  Also, from the prime number theorem and subdivision of the $m$ variable we have
$$ \sum_{2 \leq m \leq x \log x} \frac{(-1)^{\pi(m)}}{m \log m} = \sum_{n \leq x} \sum_{p_n \leq m < p_{n+1}} \frac{(-1)^{\pi(m)}}{m \log m} + o(1).$$
Thus it will suffice to show that the series
$$ \sum_{n=1}^\infty \left( \sum_{p_n \leq m < p_{n+1}} \frac{(-1)^{\pi(m)}}{m \log m} - \frac{(-1)^n n (p_{n+1}-p_n)}{p_n p_{n+1}} \right)$$
is absolutely convergent.  Since $(-1)^{\pi(m)} = (-1)^n$ for $p_n \leq m < p_{n+1}$, this is equivalent to showing that\footnote{We choose the initial index $n=10$ of this sum rather arbitrarily; any index for which $\log\log n$ is well-defined and positive would suffice here.}
\begin{equation}\label{nsum}
 \sum_{n=10}^\infty \left| \sum_{p_n \leq m < p_{n+1}} \frac{1}{m \log m} - \frac{n (p_{n+1}-p_n)}{p_n p_{n+1}}  \right| < \infty.
\end{equation}
By the intermediate value theorem we have
$$ \sum_{p_n \leq m < p_{n+1}} \frac{1}{m \log m}  = \frac{p_{n+1} - p_n}{x_n \log x_n}$$
for some $p_n \leq x_n < p_{n+1}$.  From the prime number theorem we have $p_n, p_{n+1} = n \log n \left(1 + O(\frac{\log \log n}{\log n}) \right)$ for $n \geq 10$, and thus after some calculation
$$ \frac{1}{x_n \log x_n} = \frac{n}{p_n p_{n+1}} + O\left( \frac{\log\log n}{n \log^3 n} \right)$$
for $n \geq 10$. On the other hand, from summation by parts and the prime number theorem we have
$$ \sum_{n=10}^\infty \frac{\log\log n}{n \log^3 n} (p_{n+1}-p_n) \ll 1 + \sum_{n=10}^\infty \frac{\log \log n}{n^2 \log^3 n} p_n < \infty,$$
and the claim follows.

\begin{remark}  A more careful analysis of the above argument shows that the $o(1)$ error in \eqref{nx} can be taken to be $O( \frac{\log\log x}{\log x})$ for $x \geq 10$.
\end{remark}

\section{Proof of main theorem}\label{impl-2}

We now establish Theorem \ref{main-impl}.  By summation by parts, it will suffice to establish the bound
\begin{equation}\label{pin-decay}
\sum_{n \leq x} (-1)^{\pi(n)} \ll \frac{x}{(\log\log x)^{1.1}}
\end{equation}
for all sufficiently large $x$, basically because the series $\sum_{n=10}^\infty \frac{1}{n \log n (\log\log n)^{1.1}}$ is (barely) convergent.  It is convenient to work with slightly shorter intervals than $[1,x]$.  By subdivision, it will suffice to show that
$$
\sum_{x \leq n < x + x^{1-\eps/2}} (-1)^{\pi(n)} \ll \frac{x^{1-\eps/2}}{(\log\log x)^{1.1}}$$
for all sufficiently large $x$, where $\eps$ is the constant in Conjecture \ref{hlc} (which we can assume to be small).  

Fix a sufficiently large $x$.  We express the above estimate probabilistically.  Let $\n$ be\footnote{We will denote random variables in boldface.} an integer drawn uniformly at random from the interval $[x, x + x^{1-\eps/2}]$, then the above estimate is equivalent to the assertion that
$$ \E (-1)^{\pi(\n)} \ll \frac{1}{(\log\log x)^{1.1}},$$
where $\E$ denotes expectation.

The next step is to apply the van der Corput $A$-process (or Weyl differencing), so that we work with the parity of primes in short intervals rather than long intervals.  We introduce the length scale 
\begin{equation}\label{H-def}
H \coloneqq \lfloor (\log\log x)^{4.4} \log x \rfloor,
\end{equation}
then for any $1 \leq h \leq H$ we have the approximate shift invariance
$$ \E (-1)^{\pi(\n)} = \E (-1)^{\pi(\n+h)} + O\left( \frac{1}{(\log\log x)^{10}} \right)$$
(say), hence on averaging in $h$
$$ \E (-1)^{\pi(\n)} = \E \frac{1}{H} \sum_{h=1}^H (-1)^{\pi(\n+h)} + O\left( \frac{1}{(\log\log x)^{10}} \right).$$
 Applying Cauchy--Schwarz, it thus suffices to show that
$$\E\left |\frac{1}{H} \sum_{h=1}^H (-1)^{\pi(\n+h)}\right|^2 \ll \frac{1}{(\log\log x)^{2.2}} $$
which on expanding the square and using the triangle inequality and approximate shift invariance would follow from
$$ \sum_{0 \leq h \leq H} |\E (-1)^{\pi(\n+h)-\pi(\n)}| \ll \frac{H}{(\log\log x)^{2.2}}.$$
To show this, it suffices from the choice \eqref{H-def} of $H$ to establish the estimate
\begin{equation}\label{pine}
\E (-1)^{\pi(\n+\lambda \log x)-\pi(\n)} \ll \frac{1}{\sqrt{\lambda}}
\end{equation}
for any $1 \ll \lambda \ll (\log \log x)^{4.4}$ (using the trivial bound $|\E (-1)^{\pi(\n+h)-\pi(\n)}| \leq 1$ to treat the cases where $h \leq \log x$).   The $\frac{1}{\sqrt{\lambda}}$ bound here is probably not optimal (heuristically this expression should be more like $e^{-2\lambda}$, as predicted both by the Cram\'er random model and the calculations of Gallagher \cite{gallagher}), but is what arises from our (slightly crude) methods, and suffices for our purposes (at least with the strength of Hardy--Littlewood conjecture assumed in Conjecture \ref{hlc}).  Informally, the estimate \eqref{pine} asserts that the parity of the prime counting function in an interval $(\n, \n+\lambda \log x]$ is approximately equidistributed for $\lambda$ moderately large (but not too large).

Fix $\lambda$ in the indicated range.  By rounding if necessary, we may take $\lambda \log x$ to be an integer. We now invoke a Bonferroni type inequality (in the spirit of the Brun pure sieve) to provide tractable upper and lower bounds on the random variable $(-1)^{\pi(\n+\lambda \log x)-\pi(\n)}$:

\begin{lemma}[Bonferroni type inequality]\label{bon}  For any nonnegative integers $N, r$, we have
$$ (-1)^N \leq \sum_{k=0}^r (-2)^k \binom{N}{k} $$
when $r$ is even and
$$ (-1)^N \geq \sum_{k=0}^r (-2)^k \binom{N}{k} $$
when $r$ is odd, using the convention that $\binom{N}{k}$ vanishes for $k>N$.
\end{lemma}

\begin{proof}  We just prove the first inequality, as the second is similar.  Let us write $f(r) \coloneqq \sum_{k=0}^r \binom{n}{k} (-2)^k$.  For $r$ even, routine calculation shows that $f(r+2) \geq f(r)$ when $r \leq 2N/3$ and $f(r+2) \leq f(r)$ when $r \geq 2N/3$.  Also, $f(0)=1$ and $f(r) = (-1)^N$ for $r \geq N$, thanks to the binomial theorem.  From this we conclude that $f(r) \geq (-1)^N$ for all even $r$, as desired.
\end{proof}

In view of this lemma and linearity of expectation, it suffices to show that
\begin{equation}\label{kr}
 \sum_{k=0}^r (-2)^k \E \binom{\pi(\n+\lambda \log x)-\pi(\n)}{k} \ll \frac{1}{\sqrt{\lambda}}
\end{equation}
whenever $r = (\log\log x)^{4.5}+O(1)$ (the significance of this choice being that $r$ is slightly larger than $\lambda$, while still within the range for which Conjecture \ref{hlc} can be applied).  It will be necessary to exploit the oscillation in $k$, and so we will largely refrain from using the triangle inequality in $k$ when proving \eqref{kr}, except regarding errors that enjoy a power savings in $x$, which will always be negligible in practice.

Fix $r$ in the indicated range.  Observe from the combinatorial interpretation of the binomial coefficient that
$$\binom{\pi(\n+\lambda \log x)-\pi(\n)}{k}  = \sum_{0 < h_1 < \dots < h_k \leq \lambda \log x} 1_{\n+h_1,\dots,\n+h_k \in {\mathcal P}}$$
and thus by linearity of expectation
$$\E \binom{\pi(\n+\lambda \log x)-\pi(\n)}{k}  = \sum_{0 < h_1 < \dots < h_k \leq \lambda \log x} \P( \n+h_1,\dots,\n+h_k \in {\mathcal P} ).$$
For $k \leq r \ll (\log\log x)^{4.5}$, we see from Conjecture \ref{hlc} and routine manipulations that
$$ \P( \n+h_1,\dots,\n+h_k \in {\mathcal P} ) = \frac{{\mathfrak S}({\mathcal H})}{\log^k x} + O( x^{-\eps/2} )$$
and thus the left-hand side of \eqref{kr} can be written as
$$ \sum_{k=0}^r \frac{(-2)^k}{\log^k x} \sum_{0 < h_1 < \dots < h_k \leq \lambda \log x} {\mathfrak S}({\mathcal H}) 
+ O\left( \sum_{k=0}^r \frac{(\lambda \log x)^k}{k!} 2^k x^{-\eps/2} \right)$$
where ${\mathcal H} \coloneqq \{h_1,\dots,h_k\}$.
The error term can be crudely bounded by $O( r (2 \lambda \log x)^r x^{-\eps/2} ) = O(x^{-\eps/4})$, which is acceptable.  Thus it will suffice to show that 
\begin{equation}\label{blip}
 \sum_{k=0}^r \frac{(-2)^k}{\log^k x} \sum_{0 < h_1 < \dots < h_k \leq \lambda \log x} {\mathfrak S}({\mathcal H}) 
\ll \frac{1}{\sqrt{\lambda}}.
\end{equation}
It is tempting to use Gallagher type estimates \cite{gallagher} on the mean value of singular series to control the left-hand side (which would morally obtain the sharper bound $O(e^{-2\lambda})$), but unfortunately the best known bounds, see \cite[Theorem 1.1]{kuperberg}, are only available for $k$ a little bit smaller than $(\log\log x)^{0.5}$, which is not quite good enough\footnote{More specifically, we were only able to use the estimates in \cite{kuperberg} to prove a weaker version of \eqref{pin-decay} in which the exponent $1.1$ is replaced by any exponent less than $1/4$.} in our situation.  To get around this, we will reformulate the above estimate using a version of the random sieve model from \cite{bft}, allowing one to exploit cancellations between the different values of $k$.  As in \cite[\S 1.3]{bft} (see also \cite{polya}), we define the sieve cutoff $z$ to be the largest prime number for which
$$ \prod_{p \leq z} (1-\frac{1}{p}) \leq \frac{1}{\log x},$$
then from Mertens' theorem we have $z \asymp x^{1/e^\gamma}$ and
\begin{equation}\label{pz}
 \prod_{p \leq z} (1-\frac{1}{p}) =\frac{1}{\log x} + O( x^{-1/e^\gamma})
\end{equation}
where $\gamma$ is the Euler--Mascheroni constant.  We select a residue class $\mathbf{a}_p \pmod{p}$ for each prime $p \leq z$ uniformly at random, independently in $p$, and then define the random sifted sets $\mathbfcal{S}_w$ for every $w \leq z$ by the formula
$$ \mathbfcal{S}_w \coloneqq \{ 0 < h \leq \lambda \log x: h \neq \mathbf{a}_p \pmod{p} \text{ for all } p \leq w \}.$$
As in \cite{bft}, one should think of $\mathbfcal{S}_z$ as a simplified model for the random set $\{ 0 < h \leq \lambda \log x: \n + h \in \mathcal{P} \}$.

For any tuple $0 < h_1 < \dots < h_k \leq \lambda \log x$ and any $\lambda \log x \leq w \leq z$, we have from construction that
\begin{equation}\label{direct}
\begin{split}
 \P( h_1,\dots,h_k \in \mathbfcal{S}_w ) &= \prod_{p \leq w} \P( h_1, \dots, h_k \neq \mathbf{a}_p \pmod{p} ) \\
&= \prod_{p \leq w} \left(1 - \frac{\nu_p(\{h_1,\dots,h_k\})}{p}\right) \\
&= {\mathfrak S}(\{h_1,\dots,h_k\}) \left( \prod_{p \leq w} \left(1-\frac{1}{p}\right)^k \right)
\prod_{p > w} \frac{\left(1-\frac{1}{p}\right)^k}{1-\frac{k}{p}}  
\end{split}
\end{equation}
since $\nu_p(\{h_1,\dots,h_k\}) = k$ for $p > w$.  If we restrict to the regime $k \leq r$, then we can write 
$$1-\frac{k}{p} = \exp\left(-\frac{k}{p} + O\left( \frac{k^2}{p^2} \right)\right)$$
and
$$1-\frac{1}{p} = \exp\left(-\frac{1}{p} + O\left( \frac{1}{p^2} \right)\right)$$
for $p>w$, and hence
$$ \prod_{p > w} \frac{\left(1-\frac{1}{p}\right)^k}{1-\frac{k}{p}} = \exp\left( O\left( \frac{k^2}{w} \right)\right)$$
and thus
\begin{equation}\label{direct-w}
 \P( h_1,\dots,h_k \in \mathbfcal{S}_w ) = {\mathfrak S}(\{h_1,\dots,h_k\}) \left( \prod_{p \leq w} \left(1-\frac{1}{p}\right)^k \right) \left( 1 + O\left( \frac{k^2}{w} \right)\right).
\end{equation}
In particular, from \eqref{pz} we have
\begin{equation}\label{hhh}
 \P( h_1,\dots,h_k \in \mathbfcal{S}_z ) = \frac{{\mathfrak S}(\{h_1,\dots,h_k\})}{\log^k x} + O( x^{-\eps} )
\end{equation}
for any $0 \leq k \leq r$ (shrinking $\eps$ if necessary).  We can then rewrite \eqref{blip} as
$$ \sum_{k=0}^r (-2)^k \sum_{0 < h_1 < \dots < h_k \leq \lambda \log x} \P( h_1,\dots,h_k \in \mathbfcal{S}_z ) 
\ll \frac{1}{\sqrt{\lambda}},$$
since the contribution of the $O( x^{-\eps} )$ error is negligible as before.  We can simplify the left-hand side to
$$\sum_{k=0}^r (-2)^k  \E \binom{\mathbf{S}_z}{k},$$
where $\mathbf{S}_w \coloneqq |\mathbfcal{S}_w|$ denotes the size of the sifted set at level $w$.  From Lemma \ref{bon} we have
$$\sum_{k=0}^r (-2)^k \binom{\mathbf{S}_z}{k} = (-1)^{S_z} + O\left( 2^r \binom{\mathbf{S}_z}{r} \right),$$
so it will suffice to establish the bounds
\begin{equation}\label{est-1}
\E (-1)^{\mathbf{S}_z} \ll \frac{1}{\sqrt{\lambda}}
\end{equation}
and
\begin{equation}\label{est-2}
2^r \E \binom{\mathbf{S}_z}{r} \ll \frac{1}{\sqrt{\lambda}}.
\end{equation}
We begin with \eqref{est-2}.  We can expand the left-hand side as
$$ 2^r \sum_{0 < h_1 < \dots < h_r \leq \lambda \log x} \P( h_1,\dots,h_r \in \mathbfcal{S}_z ).$$
Using \eqref{hhh} and discarding the error term using crude estimates, it suffices to show that
$$ \frac{2^r}{\log^r x} \sum_{0 < h_1 < \dots < h_r \leq \lambda \log x} {\mathfrak S}({\mathcal H}) \ll \frac{1}{\sqrt{\lambda}}.$$
Using \cite[Theorem 1.2]{kuperberg}, we may bound the left-hand side by
$$ \ll \frac{2^r \lambda^r}{r!} (3 \log r)^r$$
which by Stirling's formula can be estimated by
$$ \ll \left(\frac{O(\lambda \log r)}{r}\right)^r.$$
Since $\lambda \leq (\log\log x)^{4.4}$ and $r = (\log\log x)^{4.5}+O(1)$, one easily verifies that this quantity is $O(1/\sqrt{\lambda})$ as required.

It remains to establish \eqref{est-1}, which is an analogue of \eqref{pine} for the random sieve model $\mathbfcal{S}_z$ of the primes.  We first control the first two moments of $\mathbf{S}_w$:

\begin{lemma}[Mean and variance]\label{var}  For $\lambda \log x \leq w \leq z$, we have
\begin{equation}\label{esw}
\E \mathbf{S}_w = \lambda \log x \prod_{p \leq w} (1-\frac{1}{p}) = \frac{\lambda \log x}{e^\gamma \log w} \left(1 + O\left(\frac{1}{\log w}\right)\right)
\end{equation}
and
\begin{equation}\label{vsw}
 \Var( \mathbf{S}_w ) \ll \frac{\lambda \log x}{\log w}.
\end{equation}
\end{lemma}

\begin{proof}  From linearity of expectation we have
$$ \E \mathbf{S}_w = \sum_{0 < h \leq \lambda \log x} \P( h \in \mathbfcal{S}_w ) $$
and similarly
$$\E(\mathbf{S}^2_w - \mathbf{S}_w) = 2 \E \binom{\mathbf{S}_w}{2} = 2 \sum_{0 < h_1 < h_2 \leq \lambda \log x} \P( h_1, h_2 \in \mathbfcal{S}_w ).$$
From \eqref{direct}, \eqref{direct-w} and Mertens' theorem we have
$$ \P(h \in \mathbfcal{S}_w) = \prod_{p \leq w} (1-\frac{1}{p}) = \frac{1}{e^\gamma \log w} \left(1 + O\left(\frac{1}{\log w}\right)	 \right)$$
for all $0 < h \leq w$, which gives \eqref{esw}.  To establish \eqref{vsw}, we see from \eqref{direct-w} that
$$\P( h_1, h_2 \in \mathbfcal{S}_w ) = {\mathfrak S}(\{h_1,h_2\}) \prod_{p \leq w} (1-\frac{1}{p})^2 \left(1 + O\left(\frac{1}{w}\right)\right)
$$
for $0 < h_1 < h_2 \leq \lambda \log x \leq w$.  On the other hand, the bound 
\begin{equation}\label{2h}
 2 \sum_{0 < h_1 < h_2 \leq H} {\mathfrak S}(\{h_1,h_2\}) \leq H^2
\end{equation}
is known for $H$ large enough.  In fact there is a long history\footnote{We are indebted to Ofir Gorodetsky and Dan Goldston for these references.} of results of this type:
\begin{itemize}
\item[(i)] The left hand side of \eqref{2h} was shown to be $H^2 + O(H\log H)$ in \cite[Lemma 2]{mont}.
\item[(ii)] This was improved to $H^2 - H \log H + O(H)$ in unpublished work of Montgomery, with a full proof appearing in \cite{croft}. Note that this is already enough to establish \eqref{2h}.
\item[(iii)] The sharper asymptotic $H^2 - H\log H + (1 - \gamma - \log 2\pi) H + O_\eps(H^{1/2+\eps})$ for any $\eps>0$ was stated without proof in \cite{goldston}, 
but a proof can be found in \cite[(16), (17)]{ms} or \cite{ms-beyond}.
\item[(iv)] Assuming the Riemann hypothesis, an even sharper bound of $H^2 - H\log H + (1 - \gamma - \log 2\pi) H + O_\eps(H^{5/12+\eps})$ was established in \cite{vaughan}, who also obtained a more modest sharpening (of Vinogradov--Korobov type) to (iii) without assuming the Riemann hypothesis.
\item[(v)] In \cite{sg}, it is shown the error term in (iii) or (iv) is of magnitude $\gg H^{1/4}$ (with either choice of sign) for infinitely many $H$.
\end{itemize}
We also remark that the fact that \eqref{2h} is somewhat less than $H^2$ is closely related to the pair correlation conjecture, as well as the expected belief the error term in the prime number theorem in medium length intervals is somewhat smaller than what is predicted from the Cram\'er random model.  See \cite{ms}, \cite{ms-beyond} for further discussion.  These asymptotics are also related to unexpected biases in the distribution of consecutive primes \cite{lemke}.

From \eqref{2h} we have
$$
2 \sum_{0 < h_1 < h_2 \leq \lambda \log x} {\mathfrak S}(\{h_1,h_2\}) \leq (\lambda \log x)^2.$$
We conclude that
$$\E(\mathbf{S}^2_w - \mathbf{S}_w) \leq (\lambda \log x)^2  \prod_{p \leq w} (1-\frac{1}{p})^2 + O\left( \frac{(\lambda \log x)^2}{w} \prod_{p \leq w} (1-\frac{1}{p})^2 \right).$$ 
Combining this with \eqref{esw} and rearranging, we conclude that
$$ \Var( \mathbf{S}_w )	 \ll \E \mathbf{S}_w + \frac{(\lambda \log x)^2}{w} \prod_{p \leq w} (1-\frac{1}{p})^2$$
and the claim now follows from \eqref{esw}, Mertens' theorem, and the hypothesis $w \geq \lambda \log x$.
\end{proof}

To establish \eqref{est-1}, we will obtain a recursive inequality for the bias $\E (-1)^{\mathbf{S}_{p_n}}$.  If $\lambda \log x \leq p_n < p_{n+1} \leq z$ are consecutive primes, we observe that $\mathbfcal{S}_{p_{n+1}}$ is obtained from $\mathbfcal{S}_{p_n}$ by removing one residue class modulo $p_{n+1}$, chosen uniformly and independently of $\mathbfcal{S}_{p_n}$.  Since $p_{n+1} \geq \lambda \log x$ exceeds the diameter of $\mathbfcal{S}_{p_n}$, we conclude that after conditioning on $\mathbf{S}_{p_n}$, we have $\mathbf{S}_{p_{n+1}} = \mathbf{S}_{p_n}$ with conditional probability $1 - \mathbf{S}_{p_n}/p_{n+1}$, and $\mathbf{S}_{p_{n+1}} = \mathbf{S}_{p_n}-1$ with conditional probability $\mathbf{S}_{p_n}/p_{n+1}$.  From the law of total expectation, we conclude the identity
\begin{align*} \E (-1)^{\mathbf{S}_{p_{n+1}}} &= \E \left( 1 - \frac{\mathbf{S}_{p_n}}{p_{n+1}} \right) (-1)^{\mathbf{S}_{p_{n}}} +  \frac{\mathbf{S}_{p_n}}{p_{n+1}} (-1)^{\mathbf{S}_{p_{n}}-1}  \\
&= \E\left( 1 - \frac{2 \mathbf{S}_{p_n}}{p_{n+1}} \right) (-1)^{\mathbf{S}_{p_{n}}}.
\end{align*}
We now exploit the concentration of $\mathbf{S}_{p_n}$ around the mean $\E \mathbf{S}_{p_n}$ provided by Lemma \ref{var}. By the triangle inequality, we  have
$$ |\E (-1)^{\mathbf{S}_{p_{n+1}}}| \leq \left(1 - \frac{2 \E \mathbf{S}_{p_n}}{p_{n+1}}\right) |\E (-1)^{\mathbf{S}_{p_{n}}}| + O\left( \frac{\E |\mathbf{S}_{p_n} - \E \mathbf{S}_{p_n}|}{p_{n+1}} \right).$$
The crucial point here is that the factor $1 - \frac{2 \E \mathbf{S}_{p_n}}{p_{n+1}}$, while positive, is slightly less than one, indicating a tendency to reduce the bias as $n$ increases.

By Cauchy--Schwarz and \eqref{vsw}, we have
$$ \E |\mathbf{S}_{p_n} - \E \mathbf{S}_{p_n}| \ll \left( \frac{\lambda \log x}{\log p_{n}} \right)^{1/2}$$
while from \eqref{esw} we have
$$ \E \mathbf{S}_{p_n} = \frac{\lambda \log x}{e^\gamma \log p_n} \left(1 + O\left(\frac{1}{\log p_n}\right)\right).$$
Bounding $1 - \frac{2 \E \mathbf{S}_{p_n}}{p_{n+1}}$ by $\exp \left( - \frac{2 \E \mathbf{S}_{p_n}}{p_{n+1}} \right)$, we conclude the recursive inequality
$$ |\E (-1)^{\mathbf{S}_{p_{n+1}}}| \leq \exp\left( - \frac{2 \lambda \log x}{e^\gamma p_n \log p_n} + O\left( \frac{\lambda \log x}{p_n \log^2 p_n} \right) \right) |\E (-1)^{\mathbf{S}_{p_{n}}}| + O\left( \frac{1}{p_{n}}  \left( \frac{\lambda \log x}{\log p_{n}} \right)^{1/2} \right).$$
Iterating\footnote{Here we are basically using the discrete version of Gronwall's inequality.} this inequality starting from the trivial bound $|\E (-1)^{S_p}| \leq 1$ for $p$ the first prime greater than $\lambda \log x$, we conclude that
$$ \E (-1)^{\mathbf{S}_z} \ll \alpha_{\lambda\log x} + \sum_{\lambda \log x \leq p \leq z} \frac{\alpha_p}{p} \left( \frac{\lambda \log x}{\log p} \right)^{1/2},$$
where $\alpha_w$ is a weight of the form
$$ \alpha_w \coloneqq \exp\left( - \sum_{w \leq p < z} \left( \frac{2 \lambda \log x}{e^\gamma p \log p} + O\left( \frac{\lambda \log x}{p \log^2 p} \right) \right) \right).$$
From the prime number theorem and summation by parts, we have
$$ \sum_{w \leq p < z} \frac{1}{p \log p} = \frac{1}{\log w} - \frac{1}{\log z} + O\left( \frac{1}{\log^2 w}\right)$$
and
$$ \sum_{w \leq p < z} \frac{1}{p \log^2 p} \ll \frac{1}{\log^2 w}$$
for any $\lambda \log x \leq w \leq z$, thus
\begin{equation}\label{alphaw}
 \alpha_w = \exp\left( -\frac{2 \lambda \log x}{e^\gamma \log w} + \frac{2 \lambda \log x}{e^\gamma \log z} + O\left( \frac{\lambda \log x}{\log^2 w} \right) \right).
\end{equation}
In particular we have
$$ \alpha_{\lambda\log x} \leq \exp\left( - c \frac{\log x}{\log \log x} \right)$$
for some absolute constant $c>0$, which is certainly $O(1/\sqrt{\lambda})$.  It thus remains to establish the bound
$$ \sum_{\lambda \log x \leq p \leq z} \frac{\alpha_p}{p} \left( \frac{\lambda \log x}{\log p} \right)^{1/2} \ll \frac{1}{\sqrt{\lambda}}.$$
We first dispose of the contribution of the small primes $p \leq x^{\frac{1}{100 \log\log x}}$.  From \eqref{alphaw} we have
$$ \alpha_p \leq \exp( - 10 \log\log x ) = \log^{-10} x$$
(say), while we may crudely bound
$$ \left( \frac{\lambda \log x}{\log p} \right)^{1/2} \ll \log x$$
for some absolute constant $c>0$ in this case.  From Mertens' theorem, the contribution of this case is thus $O( \log^{-8} x )$, which is certainly $O(1/\sqrt{\lambda})$.  Thus it remains to show that
\begin{equation}\label{man}
 \sum_{x^{\frac{1}{100 \log\log x}} \leq p \leq z} \frac{\alpha_p}{p} \left( \frac{\lambda \log x}{\log p} \right)^{1/2} \ll \frac{1}{\sqrt{\lambda}}.
\end{equation}
For each prime $p$ in this range we have
\begin{equation}\label{loglog}
 m \leq \frac{\lambda \log x}{\log p} < m+1 
\end{equation}
for some natural number $m$ with
$$ \frac{\lambda \log x}{\log z} - 1 \leq m \ll (\log \log x)^{O(1)}.$$
For $p$ in such a range, the $\frac{\lambda \log x}{\log^2 p}$ term in \eqref{alphaw} is bounded, hence
$$\alpha_p \ll \exp\left( - \frac{2}{e^\gamma} \left( m - \frac{\lambda \log x}{\log z} \right) \right)$$
and so the contribution of a given $m$ to \eqref{man} is
$$ \ll \exp\left( - \frac{2}{e^\gamma} \left( m - \frac{\lambda \log x}{\log z} \right) \right) m^{1/2} \sum_{p: m \leq \frac{\lambda \log x}{\log p} < m+1} \frac{1}{p}.$$
The restriction \eqref{loglog} restricts $\log\log p$ to an interval of size $O(1/m)$, but also (due to the upper bound on $m$) also encompasses at least one dyadic range $[X,2X]$ of $p$.  By Mertens' theorem (or the Brun--Titchmarsh inequality), we thus have
$$ \sum_{p: m \leq \frac{\lambda \log x}{\log p} < m+1} \frac{1}{p} \ll \frac{1}{m},$$
and so the total contribution of \eqref{man} is
$$ \ll \sum_{m \geq \frac{\lambda \log x}{\log z} - 1} \frac{1}{m^{1/2}} \exp\left( - \frac{2}{e^\gamma} \left( m - \frac{\lambda \log x}{\log z} \right) \right).$$
This is an exponentially decaying sum in $m$, so can be controlled by its first term
$$ \ll \left(\frac{\lambda \log x}{\log z} \right)^{-1/2} \asymp \frac{1}{\lambda^{1/2}},$$
giving the desired claim.

\begin{remark} The above analysis in fact shows that the convergence rates for the partial sums $\sum_{n \leq x} \frac{(-1)^n n}{p_n}$, $\sum_{10 \leq n \leq x} \frac{(-1)^{\pi(n)}}{n \log n}$ are $O
\left( \frac{1}{(\log\log x)^{0.1}}\right)$, assuming Conjecture \ref{hlc}.  Probabilistic heuristics suggest that the true rate of convergence is $O(1/\log x)$ and $O(1/x^{1/2-o(1)})$ respectively for the above two partial sums, though conjecturally the convergence of the first partial sum should also become $O(1/x^{1/2-o(1)})$ when one averages together consecutive partial sums to reduce oscillation.
\end{remark}

\section{A generalization}\label{general-sec}

The questions and arguments in this section arose from several conversations with William Banks.

In this section we record a generalization of the above results, in which the base $(-1)$ is replaced by other non-trivial unit phases.  More precisely, we consider the following two questions:

\begin{question}\label{main-z}  Let $z \in \C$ with $|z|=1$ and $z \neq 1$. Does the series $\sum_{n=1}^\infty \frac{z^n n}{p_n}$  converge?
\end{question}

\begin{question}\label{main-2-z}  Let $z \in \C$ with $|z|=1$ and $z \neq 1$. Does the series $\sum_{n=2}^\infty \frac{z^{\pi(n)}}{n \log n}$ converge?
\end{question}

We first observe that these questions continue to be equivalent for any fixed $z$.  Indeed one has a generalization
$$ \sum_{n \leq x} \frac{z^n n}{p_n} = \frac{z}{z-1} \sum_{2 \leq m \leq x \log x} \frac{z^{\pi(m)}}{m \log m} + C_z + o(1)$$
of \eqref{nx} as $x \to \infty$, where $C_z$ is a constant depending on $z$ and the rates of decay in asymptotic notation such as $o(1)$ is permitted to depend non-uniformly on $z$.  We sketch the argument as follows.  Writing $z^n = \frac{z}{z-1} (z^n - z^{n-1})$ and summing by parts, we have
$$ \sum_{n \leq x} \frac{z^n n}{p_n}  = -\frac{z}{2(z-1)} + \frac{z}{z-1} \sum_{n \leq x} z^n \left( \frac{n}{p_n} - \frac{n+1}{p_{n+1}} \right) + o(1).$$
As before we split
$$ z^n \left( \frac{n+1}{p_{n+1}} - \frac{n}{p_n} \right) = z^n \frac{n(p_{n+1}-p_n)}{p_n p_{n+1}} - \frac{z^n}{p_n}.$$
A further summation by parts reveals that $\sum_{n=1}^\infty \frac{z^n}{p_n}$ converges, while from \eqref{nsum} we see that
$$ \sum_{n=1}^\infty \left( \sum_{p_n \leq m < p_{n+1}} \frac{z^{\pi(m)}}{m \log m} - \frac{z^n n (p_{n+1}-p_n)}{p_n p_{n+1}} \right)$$
is absolutely convergent.  The claim then follows along similar lines to the remaining arguments in Section \ref{impl}.

Theorem \ref{main-impl} can also be extended to general $z$ by adapting the arguments in Section \ref{impl-2}.  The Bonferroni-type inequality in Lemma \ref{bon} can be replaced with the weaker bound
$$ z^N = \sum_{k=0}^r (z-1)^k \binom{N}{k} + O_z \left( r 2^r \binom{N}{r} \right)$$
for any $r$, which can be established by a similar argument (but using the triangle inequality instead of alternating series bounds).  The rest of the argument proceeds with minor modifications; for instance, several factors of $2$ appearing in expressions such as $2 {\bf E} {\bf S}_{p_n}$ will need to be by replaced by $1-z$ (or $\mathrm{Re}(1-z)$ after an application of the triangle inequality).  We leave the details of these changes to the interested reader.

To summarize the above discussion, we have shown that assuming Conjecture \ref{hlc}, the power series $\sum_{n=1}^\infty \frac{z^n n}{p_n}$ converges everywhere on the boundary of the disk $\{ z: |z| \leq 1\}$ of convergence, except at $z=1$ where of course it diverges thanks to the prime number theorem.  Based on analogy with random power series, it seems natural to conjecture\footnote{We thank an anonymous commenter on the author's blog for this question.} that the unit circle is in fact the natural boundary of this series.  

\section{Further questions}

In \cite{en}, the divergence and convergence of various other series relating to the primes is considered.  In particular, the paper demonstrates the convergence of
\begin{equation}\label{pn1}
 \sum_{n=10}^\infty \frac{1}{n (\log\log n)^c (p_{n+1}-p_n)}
\end{equation}
whenever $c>2$, and conjectures divergence for $c=2$; the paper also mentions Question \ref{main}, as well as the question of determining the convergence of the series
\begin{equation}\label{pn2}
 \sum_{n=1}^\infty \frac{(-1)^n}{p_{n+1}-p_n}
\end{equation}
and
\begin{equation}\label{pn3}
 \sum_{n=1}^\infty \frac{(-1)^n}{n(p_{n+1}-p_n)}
\end{equation}

The divergence of \eqref{pn2} can now be established unconditionally, since Zhang's theorem \cite{zhang} on bounded gaps between primes implies that the summand fails to converge to zero.  The divergence of \eqref{pn1} when $c=2$ can be deduced from Conjecture \ref{hlc} by an argument that we sketch\footnote{We thank Kevin Ford for supplying a version of this argument.} as follows.  By dyadic decomposition and the prime number theorem, it suffices to establish the lower bound
$$ \sum_{X \leq p_n < 2X} \frac{1}{p_{n+1}-p_n} \gg \frac{X \log\log X}{\log^2 X}$$
for all large $X$, since $\frac{1}{n (\log\log n)^2} \asymp \frac{\log X}{X (\log \log X)^2}$ for all $n$ in this range, and $\frac{1}{\log X \log\log X}$ diverges when summed over powers of two.  By further dyadic decomposition in the gap $p_{n+1}-p_n$, it suffices to obtain the bound
\begin{equation}\label{lam-strong}
 \# \{ X \leq p_n < 2X: p_{n+1} - p_n \leq \lambda \log X \} \gg \frac{\lambda X}{\log X}
\end{equation}
for any $\frac{2}{\log X} \leq \lambda \leq 1$. It suffices to treat the case where $\lambda$ is small.  The left-hand side can be written as
$$ \sum_{X \leq p_n < 2X} 1_{\# (p_n, p_n + \lambda \log X] \cap {\mathcal P} \neq 0}.$$
Using the Bonferroni inequalities
$$ m - \binom{m}{2} \leq 1_{m \neq 0} \leq m$$
for any $m \geq 0$, and rearranging, we can write this expression as
$$ \sum_{2 \leq h \leq \lambda \log X} \sum_{X \leq n < 2X} 1_{n, n+h \in {\mathcal P}}
- O\left( \sum_{2 \leq h < h' \leq \lambda \log X} \sum_{X \leq n < 2X} 1_{n, n+h, n+h' \in {\mathcal P}} \right)
$$
so by Conjecture \ref{hlc} it suffices to establish the bounds
$$ \sum_{2 \leq h \leq \lambda \log X} {\mathfrak S}(\{0,h\}) \asymp \lambda \log X$$
and
$$ \sum_{2 \leq h < h' \leq \lambda \log X} {\mathfrak S}(\{0,h,h'\}) \asymp \lambda^2 \log^2 X$$
since the contribution of the error terms are easily treated.  But this follows from Gallagher's estimate \cite[(3)]{gallagher} (or from the more refined estimates in \cite{ms}, \cite{kuperberg}).  We remark that in the absence of the prime tuples conjecture, it should be possible to obtain the weaker version
$$ \# \{ X \leq p_n < 2X: p_{n+1} - p_n \leq \lambda \log X \} \gg \frac{\lambda^{49} X}{\log X}$$
of \eqref{lam-strong} for $\frac{246}{\log X} \leq \lambda \leq 1$ by combining the method of Maynard \cite{maynard-dense} with the calculations in \cite{selberg}.  This is however insufficient to settle the divergence of \eqref{pn2} in the $c=2$ case.  We also remark that a matching upper bound to \eqref{lam-strong} can be established from standard upper bound sieves.

The convergence of the sum \eqref{pn3} remains a challenging open problem, even with the assumption of Conjecture \ref{hlc}; a significant portion of this sum arises from the small prime gaps in which $p_{n+1}-p_n$ is much smaller than $\log n$, but these are a sparse set of $n$ and so the van der Corput type methods used in this paper do not appear to be effective to induce cancellation from the $(-1)^n$ terms on this sparse set. However, standard probabilistic heuristics (for instance, replacing $(-1)^n$ by independent random signs and invoking Khintchine's inequality) suggest rather strongly that this series should be convergent.  Indeed this heuristic suggests the stronger assertion that the series
$$
 \sum_{n=1}^\infty \frac{(-1)^n}{n^\theta(p_{n+1}-p_n)}$$
should converge for $\theta > 1/2$ and diverge for $\theta \leq 1/2$.

\end{document}